\documentclass[10pt]{article}
\usepackage[top=1.5in, bottom=1.5in, left=1.35in, right=1.35in]{geometry}
\usepackage{amsxtra,amsmath,amscd,amssymb,amsthm}
\usepackage{mathrsfs}

\usepackage{longtable} 
\usepackage{verbatim}
\usepackage{xspace}
\usepackage{ifthen}
\usepackage{tabularx}

\usepackage{setspace,graphicx}
\usepackage{multirow,xcolor}
\usepackage{color, colortbl,enumerate}
\usepackage{tikz}
\usepackage{tikz-cd}
\usepackage{leftidx}
\usepackage[titletoc]{appendix}
\usepackage{cite}
\usepackage{floatrow}
\newfloatcommand{capbtabbox}{table}[][\FBwidth]
\newcommand{\quash}[1]{}  

\theoremstyle{plain}
\numberwithin{equation}{section}

\newtheorem{thm}[equation]{Theorem}
\newtheorem{prop}[equation]{Proposition}
\newtheorem{lm}[equation]{Lemma}
\newtheorem{defi}[equation]{Definition}

\newtheorem{rk}[equation]{Remark}

\newtheorem{ex}[equation]{Example}

\theoremstyle{remark}

\newcommand{\Gortz}{G\"{o}rtz}

\renewcommand{\phi}{\varphi}

\newcommand{\Isom}{{\rm Isom}}
\newcommand{\End}{{\rm End}}
\newcommand{\Hom}{{\rm Hom}}
\newcommand{\Aut}{{\rm Aut}}
\newcommand{\Res}{{\rm Res}}
\newcommand{\Gal}{{\rm Gal}}
\newcommand{\Spec}{{\rm Spec}}

\newcommand{\etale}{\' etale }
\newcommand{\Int}{{\rm Int}}

\newcommand{\Gr}{{\rm Gr}}
\newcommand{\Alg}{{\rm Alg}}

\newcommand{\Sets}{{\rm Sets}}
\newcommand{\Spin}{{\rm Spin}}
\newcommand{\GL}{{\rm GL}}

\newcommand{\SO}{{\rm SO}}

\newcommand{\SSO}{{\rm\bf SO}}

\newcommand{\SSpin}{{\rm\bf Spin}}

\newcommand{\ZZ}{\mathbb{Z}}

\newcommand{\DD}{\mathbb{D}}
\newcommand{\LL}{\mathbb{L}}

\newcommand{\calX}{\mathcal{X}}
\newcommand{\calO}{\mathcal{O}}
\newcommand{\calE}{\mathcal{E}}
\newcommand{\calF}{\mathcal{F}}

\newcommand{\calM}{\mathcal{M}}

\newcommand{\scrG}{\mathscr{G}}

\newcommand{\al}{\alpha}

\newcommand{\la}{\lambda}

\newcommand{\lx}{\leftidx}
\newcommand{\bb }{\langle}
\newcommand{\pp}{\rangle}
\newcommand{\llp}{(\!(}
\newcommand{\rlp}{)\!)}
\newcommand{\lp}{[\![}
\newcommand{\rp}{]\!]}

\newcommand{\mathleft}{\@fleqntrue\@mathmargin0pt}
\newcommand{\mathcenter}{\@fleqnfalse}

\newcommand{\Addresses}{{
		\bigskip
		\footnotesize
		
	 \textsc{Dept. of Mathematics, Michigan State University,
		East Lansing, MI, 48824, USA}\par\nopagebreak
		\textit{E-mail address:} \texttt{zhaozhi8@msu.edu}
}}		
\begin{document}
\bibliographystyle{plain}

\title{Affine Grassmannians for Triality Groups}
\date{}
\author{Zhihao Zhao}
\maketitle

\begin{abstract}
We study affine Grassmannians for ramified triality groups. These groups 
are of type ${}^3D_4$, so they are forms of the orthogonal or the spin groups in $8$ variables. They can be given as automorphisms of certain twisted composition algebras obtained from the octonion algebra. Using these composition algebras, we give descriptions of the affine Grassmannians for these triality groups as functors classifying suitable lattices in a fixed space.
\end{abstract}

\tableofcontents
\section{Introduction}

Affine Grassmannians can be defined by loop groups. Let $k$ be a field, and let $G_0$ be an algebraic group over $\Spec(k)$. We consider the functor $LG_0$ on the category of $k$-algebras,
\[
R\mapsto LG_0(R)=G_0(R\llp t\rlp),
\] 
where $R\llp t\rlp$ is the Laurent power series with variable $t$. This functor is represented by an ind-scheme, called the loop group associated to $G_0$. We consider the positive loop group $L^+G_0$, which is the functor on the category of $k$-algebras,
\[
R\mapsto L^+G_0(R)=G_0(R\lp t\rp).
\]
Then $L^+G_0\subset LG_0$ is a subgroup functor, and the fpqc-quotient $\Gr_{G_0}=LG_0/L^+G_0$ is by definition the affine Grassmannian. The fpqc-sheaf $\Gr_{G_0}$ is also represented by an ind-scheme. We refer \cite{BL}, \cite{BD} for important results on the structure of loop groups and associated affine Grassmannians. These results have applications to the theory of reduction and local models of Shimura varieties, \cite{G1},\cite{Go2},\cite{PZ}.

In \cite{PR3}, Pappas and Rapoport developed a similar theory of twisted loop groups and of their associated affine Grassmannians: Let $G$ be a linear algebraic group over $k\llp t\rlp$. Consider the twisted loop group $LG$, which is also the ind-scheme representing the functor:
\[
LG: R\rightarrow G(R\llp t\rlp)
\]
for any $k$-algebra. Since $R\llp t\rlp$ is a $k\llp t\rlp$-algebra, the definition makes sense. Notice that when $G=G_0\otimes_k k\llp t\rlp$, we recover the previous definition in the untwisted case. To define the positive loop group, we assume $G$ is reductive, and choose a parahoric subgroup $\scrG$ of $G$. This is a smooth group scheme over $k\lp t\rp$ with $\scrG\otimes_{k\lp t\rp}k\llp t\rlp=G$. Then the positive loop group $L^+\scrG$ is an infinite-dimensional affine group scheme $L^+\scrG$ representing the functor:
\[
L^+\scrG: R\rightarrow \scrG(R\lp t\rp)
\]
The fpqc-quotient $\Gr_\scrG=LG/L^+\scrG$ is represented by an ind-scheme over $k$, which we call the affine Grassmannian associated to $G$. In this paper, we take the twisted loop group
to be a ramified triality group and study the corresponding affine Grassmannian. Our main goal (see Theorem \ref{thm 11} below) is an explicit description of the triality affine Grassmannian in terms of lattices with extra structure. Lusztig [10] first showed that affine Grassmannians for simple Lie algebras can be described in terms of certain orders, which are lattices closed under the Lie bracket. Here we aim for an explicit description in terms of lattices in the standard representation
which is more in line with such descriptions known for classical groups. For example, Pappas and Rapoport gave such descriptions for affine Grassmannians and affine flag varieties for unitary groups in \cite{PR3} using lattices (or lattice chains) which are self-dual for a hermitian form.  See also work of \Gortz \cite{Go2} and of Smithling \cite{SM2} for the symplectic and the split orthogonal groups, respectively. It turns out that the case of the ramified triality group, which we consider here, is considerably more complicated.

What are triality groups? Let $G_0$ be an adjoint Chevalley group of type $D_4$ over a field $F_0$. Consider the Dynkin diagram:
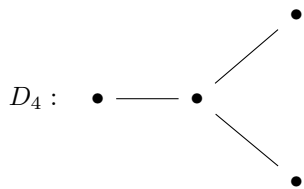
\[
\begin{tikzcd}[arrows=-]
&&\bullet\arrow[dl]\\
D_4: \quad \bullet \arrow[r] & \bullet \arrow[dr] &\\
 &&\bullet
 \end{tikzcd}
\]
The Dynkin diagram of type $D_4$ has a symmetry not shared by other Dynkin diagrams: it admits automorphisms of order 3. 
 Since the automorphism of the Dynkin diagram of type $D_4$ is isomorphic to the symmetric group $S_3$, there is a split exact sequence of algebraic groups:
\[
1\rightarrow G_0 \rightarrow \Aut(G_0)\stackrel{f}{\rightarrow} S_3\rightarrow 1.
\]
Thus, $G_0$ admits outer automorphisms of order 3, which we call trialitarian automorphisms. The fixed elements of $G_0$ under such an outer automorphism, define groups of type $G_2$:
\[
G_2\quad \bullet\Lleftarrow \bullet
\]
Consider the Galois cohomology set $H^1(F_0,\Aut(G_0)):= H^1(\Gamma_0, \Aut(G_0))$, where $\Gamma_0$ is the absolute Galois group $\Gal(F_{0,sep}/F_0)$. Adjoint algebraic groups of type $D_4$ over $F_0$ are classified by $H^1(F,\Aut(G_0))$ (29.B, \cite{KMRT}), and the map induced by $f$ in cohomology $f^1: H^1(F_0, \Aut(G_0))\rightarrow H^1(F_0, S_3)$ associates to $G_0$ of type $D_4$ the isomorphism class of a cubic \etale $F_0$-algebra $F$, see \cite{KMRT}. The possibilities of $F$ are summarized as follows:
\[
\begin{tabular}{c|c}
$F$ & type $G_0$  \\
\hline
$F_0 \times F_0 \times F_0$ & $\leftidx^1 D_4$  \\
$F_0 \times \Delta$ & $\leftidx^2 D_4$  \\
Galois field ext. & $\leftidx^3 D_4$  \\
non-Galois field ext. & $\leftidx^6 D_4$
\end{tabular}
\]
The group $G_0$ is said to be of type $\leftidx^1 D_4$ if $F$ is split, of type $\leftidx^2 D_4$ if $F \cong F_0 \times \Delta$ for some quadratic separable field extension $\Delta / F_0$, of type $\leftidx^3 D_4$ if $F$ is a cyclic field extension over $F_0$, and of type $\leftidx^6 D_4$ if $F$ is a non-cyclic field extension. In our paper, we consider the $\leftidx^3 D_4$ case and call the corresponding $G_0$ the triality group. 

These triality groups are often studied by composition algebras. By composition algebras, we mean algebras (not necessarily associative) with a nonsingular quadratic form $q$ such that $q(x\cdot y)=q(x)q(y)$ for all $x$, $y$ in this algebra. We give a review of symmetric composition algebras and normal twisted composition algebras in \S 2. Composition algebras can be used to describe exceptional groups. For example, Springer shows the automorphism of an octonion algebra is of type $G_2$ (\S2.3, \cite{Sp}). Here the octonion algebra is an 8-dimensional composition algebra. We can view this automorphism group as the fixed subgroup of a spin group of an octonion algebra under outer automorphisms (Proposition 35.9, \cite{KMRT}). In \S 3, we extend this result and show that the subgroup of a spin group of a normal twisted composition algebra, which is fixed under outer automorphisms, is of type $\leftidx^3 D_4$. This will be our main tool to study affine Grassmannians for triality groups.

Let $k$ be a field with characteristic char$(k)\neq 2,3$. We assume the cubic primitive root $\xi$ is in $k$. We set $F_0=k\llp t\rlp$ (resp. $F=k\llp u\rlp$) the ring of Laurent power series, with ring of integers $k\lp t\rp$ (resp. $k\lp u \rp$). Set $u^3=t$ so that $F/F_0$ is a cubic Galois extension with $\Gal(F/F_0)=\bb\rho \pp\cong A_3$, where $\rho$ acts on $u$ by $\rho(u)=\xi u$. In \S 2, we define the normal twisted composition algebra $(V,*)$ over $F$. Here $(V,*)$ is a 8-dimensional vector space with an $F_0$-bilinear product $*$ and a nonsingular quadratic form $q$ satisfying certain properties (see Definition \ref{def 21}).  We also fix a finitely generated  projective $k\lp u\rp$-module $\LL$ in $V$, which we call the standard lattice in $V$.  Denote by $\bb\ ,\ \pp $ the bilinear form associated to $q$. We show that the spin group $\SSpin(V,*)$ over $F$ has an outer automorphism, and the subgroup of $\Res_{F/F_0}\SSpin(V,*)$, which is fixed under the outer automorphism, is the triality group $G$ we are interested in, i.e.,
\[
G=\Res_{F/F_0}\SSpin(V,*)^{A_3}.
\]
We now choose the parahoric group scheme $\scrG$ over $\Spec(k\lp t\rp)$ given by the lattice $\LL$. This is a ``special" parahoric subgroup in the sense of Bruhat-Tits theory. Recall the generic fiber $\scrG_\eta$ of $\scrG$ is equal to $G$.  The quotient fpqc sheaf  $L\scrG_\eta/L^+\scrG$  is by definition the affine Grassmannian for the triality group over $\Spec(k)$. Our main theorem is:
\begin{thm}\label{thm 11}
There is an $L\scrG_{\eta}$-equivariant isomorphism 
\[
L\scrG_{\eta}/L^+\mathscr{G}\simeq \mathscr{F}
\]	
where the functor $\mathscr{F}$ sends a $k$-algebra $R$ to the set of finitely generated projective $R\lp u\rp $-modules $L$ (i.e., $R\lp u\rp $-lattices) of $V\otimes_{k}R\cong R\llp u\rlp^8$, such that
\begin{itemize}
	\item [(1)] $L$ is self dual under the bilinear form $\bb\ ,\ \pp$, i.e., $L\simeq \Hom_{R\lp u\rp}(L, R\lp u\rp )$.
	\item [(2)] $L$ is closed under multiplication, $L*L\subset L$.
	\item [(3)]  There exists $a\in L$, such that $q(a)=0$,  $\bb  a*a,a\pp=1$.

	\item [(4)] For $a$ as in (3), let $e=a+a*a$. Then, we have $\overline{e*x}=-\bar{x}=\overline{x*e}$ for any $\bar{x}$ satisfying $\bb  \bar{x},\bar{e}\pp =0$. (Here, $\bar{x}$ is the image of $x$ under the canonical map $L\rightarrow L/uL$.)\end{itemize}
\end{thm}

In particular, it gives a bijection between $k$-points in the affine Grassmannian for the triality group and a certain set of $k\lp u\rp$-lattices in $V$ that satisfy some special conditions. The  proof of this theorem is inspired by the construction of twisted composition algebras by Springer in \cite[\S 4.5]{Sp}: Observe that every normal twisted composition algebra (with isotropic quadratic form) $(V,*)$ contains a special hyperplane $V_0:=F f_1\oplus F f_2$, where the sum $f_1+f_2$ is the para-unit. 
Set $V_1:=V_0*f_1$ and $V_2:=V_0*f_2$. We can decompose $V$ as $V=V_0\oplus V_1\oplus V_2$. 
In \S5, we use similar ideas and extend Springer's results to suitable lattices: To decompose a lattice $L$ in $V$, we need $L$ is closed under multiplication, self dual, and contains the elements $a, a*a$ that play similar roles as $f_1, f_2$ in $V$ (see Theorem \ref{thm 11} for details).  When $L$ satisfies the conditions (1)-(4) in Theorem \ref{thm 11}, we can decompose $L$  as $L=R\lp u\rp a\oplus R\lp u\rp(a*a) \oplus L_1\oplus L_2$. Furthermore, there exist a basis of $L_1$ (resp. $L_2$) such that the multiplication table of $L$ is the same as the standard lattice $\LL$. Thus, we can define a morphism $g$ in $LG$ such that $L=g(\LL)$.

The triality group $G$ we consider in this paper is simple and simply connected as a form of the spin group. We can also consider variants $G'$ with the same adjoint group $G'_{\rm ad}\simeq G_{\rm ad}$ and use similar ideas. These variants and their associated (global) affine Grassmannians are useful for describing corresponding local models,  in the sense of \cite{PZ}. Indeed, by fixing a coweight $\mu$ of $G'$, we can obtain a   description of the corresponding local model as classifying lattices whose distance from the standard lattice is ``bounded by $\mu$". We will describe this in another work.  

The results in this paper are part of my thesis at Michigan State University. I would like to thank my advisor G. Pappas for his useful suggestions and patient help. 
\section{Twisted composition algebras}

Twisted composition algebras were introduced by Springer in his 1963 lecture notes \cite{Sp}, to get a new description of Albert algebras. We recall the definition from \cite{Sp} and \cite{KT}. Let $F_0$ be a field with char($F_0)\neq 2,3$, and let $F$ be a separable cubic field extension of $F_0$. The normal closure of $F$ over $F_0$ is $F'=F(d)$, where $d$ satisfies a separable quadratic equation over $F_0$ (see Theorem 4.13, \cite{JA1}). We can take $d=\sqrt D$, the square root of the discriminant $D$ of $F$ over $F_0$. We set $F_0'=F_0(d)$. So either $F$ is the Galois extension of $F_0$ with cyclic Galois group of order 3, and then $F'=F, F_0'=F_0$; or $F'$ and $F_0'$ are quadratic extensions of $F$ and $F_0$, respectively, and $F'$ is the Galois extension of $F_0'$. We will focus on the case that the separable cubic extension $F/F_0$ is also normal, and call algebras of this type ``normal twisted composition algebras".

Let $F/F_0$ be a cubic Galois extension. We set $\Gamma=\Gal(F/F_0)$, with $\rho$ the generator of $\Gamma$. Set $\theta=\rho^2$, then $\Gamma=\{1,\rho,\theta\}$.

\begin{defi}\label{def 21}
A normal twisted composition algebra (of dimension 8) is a 5-tuple $(A,F,q,\rho,*)$, where $A$ is a vector space of dimension 8 over $F$ with a nonsingular quadratic form $q$, and associated bilinear form $\bb  ,\pp $. We have an $F_0$-bilinear product $*: A\times A\rightarrow A$ on $F$ with the following properties:
\begin{itemize}
	\item [(1)]  The product $x*y$ is $\rho$-linear in $x$ and $\theta$-linear in y, that is:
\[
(\lambda x)*y=\rho(\lambda)(x*y),~~x*(\lambda y)=\theta(\lambda)(x*y),
\]
	\item [(2)] We have $q(x*y)=\rho(q(x))\theta(q(y))$,
	\item [(3)] We have $\bb  x*y,z\pp =\rho(\bb  y*z,x\pp )=\theta(\bb  z*x,y\pp )$
\end{itemize}
for all $x,y,z\in A$,and $\lambda\in F$.
\end{defi}

Let $A'=(A',F,q',\rho',*')$ be another normal twisted composition algebra. A similitude $A\rightarrow A'$ is defined to be an $F$-linear isomorphism $g:A\rightarrow A'$, for which there exists $\lambda\in F^*$, such that 
\[
q'(g(x))=\rho(\lambda)\theta(\lambda)q(x), ~~g(x)*'g(y)=\lambda g(x*y),
\]
for all $x,y\in A$. We denote by $A'=A_\la$. The scalar $\lambda$ is called the multiplier of the similitude. Similitudes with multiplier 1 are called isometries. We will use the following lemmas which their proofs can be found in \cite[Lemma 4.1.2, Lemma 4.1.3]{Sp}.

\begin{lm}\label{lm 232}
 We have
\begin{itemize}
	\item [(1)] $\bb  x*z, y*z\pp =\rho(\bb  x,y\pp )\theta(q(z))$,
	\item [(2)] $\bb  z*x, z*y\pp =\theta(\bb  x,y\pp )\rho(q(z))$,
	\item [(3)] $\bb  x*z,y*w,\pp + \bb  x*w,y*z\pp =\rho(\bb  x,y\pp )\theta(\bb  z,w\pp )$,
\end{itemize}
for all $x,y,z,w\in A$.
\end{lm}

\begin{lm}\label{lm 233}
We have
\begin{itemize}
	\item [(1)] $x*(y*x)=\rho(q(x))y$,\quad $(x*y)*x=\theta(q(x))y$,
	\item [(2)] $x*(y*z)+z*(y*x)=\rho(\bb  x,z\pp )y$,\quad $(x*y)*z+(z*y)*x=\theta(\bb  x,z\pp )y$,
	\item [(3)] $(x*x)*(x*x)=T(x)x-q(x)(x*x)$, where $T(x):=\bb  x*x,x\pp  \in F_0$,
\end{itemize}
for all $x,y,z\in A$.
\end{lm}

\begin{rk}{\rm 
View a normal twisted composition algebra $(A,F,q,\rho,*)$ as an 8-dim quadratic space. We can discuss isotropic subspaces in $A$.	 An element $x\in A$ is called isotropic if $q(x)=0$. A subspace $W$ of $A$ is said to be isotropic if $q(x)=0$ for all $x\in W$. A maximal isotropic subspace is an isotropic subspace with the maximal dimension. All maximal isotropic subspaces have the same dimension, which is called the Witt index of $q$. The index is at most equal to half dimension of the vector space, which in our case, is 4.
}\end{rk}

It turns out that a normal twisted composition can be obtained by scalar extension from a symmetric composition algebra. Recall from \cite{KMRT}, \S34 that a symmetric composition algebra (of dimension 8) is a  triple $(S,\star,q)$, where $(S,q)$ is an 8-dimensional $F_0$-quadratic space (with nondegenerate bilinear form $\bb,\pp$) and $\star: S \times S \rightarrow S$ is a bilinear map such that for all $x, y, z \in S$,
\[
q(x\star y)=q(x)q(y),\quad \bb x\star y, z\pp=\bb x,y\star z\pp.
\]
By \cite[Lemma 34.1]{KMRT}, the above definition is equivalent to
\[
x\star(y\star x)=q(x)y=(x\star y)\star x, \quad\text{for all}~x,y\in S.
\]
Given a symmetric composition $(S,\star,q)$ over $F_0$. We can get a normal twisted composition algebra $\tilde{S}=S\otimes (F,\rho)$ as follows:
\[
S\otimes (F,\rho):=(S\otimes_{F_0}F, F, q_F, \rho, *)
\]
where $q_F$ is the scalar extension of $q$ to $F$ and $*$ is defined by extending $\star$ linearly to $S\otimes_{F_0} F$ and setting
\[
x*y=(id_S\otimes \rho) (x)\star (id_S\otimes \theta)(y) ~\text{~for all}~x,y\in S\otimes_{F_0} F
\]
(see \cite{KT}). A normal twisted composition algebra $A$ over $F$ is said to be reduced if there exist a symmetric composition algebra $S$ over $F_0$ and $\lambda\in F^*$ such that $A$ is isomorphic to  $\tilde{S}_{\lambda}$.

\begin{ex}
{\rm The main tool that we use in this paper is the normal twisted composition algebra obtained from the para-Cayley algebra. Consider the Cayley (octonion) algebra $(C,\diamond)$ over $F_0$ with the bilinear form $\bb,\pp$ and the conjugacy map: 
\[
r(x):=\bb x,e\pp e-x,\quad \text{for}~ x\in A.
\]
Consider the new multiplication $\star: x\star y:=r(x)\diamond r(y)$. This new multiplication yields $\bb x\star y, z\pp=\bb x,y\star z\pp$. So $(C,\star)$ is a symmetric composition algebra, which is called the para-Cayley algebra (see \cite[34.A]{KMRT}). The multiplication table of $(C,\star)$ is given by Table \ref{tab 1} (we write ``$\cdot$" instead of $0$ for clarity); and denote by $(V,*)$ the normal twisted composition algebra obtained from the para-Cayley algebra.

\begin{table}[H]
\begin{tabular}{c|c|cccc|cccc}
\multicolumn{9}{c}{y} \\ \hline
		&$\star$&$e_1$&$e_2$&$e_3$&$e_4$&$e_5$&$e_6$&$e_7$&$e_8$\\ \hline
		&$e_1$&$\cdot$&$\cdot$&$\cdot$&$-e_1$&$\cdot$&$-e_2$&$e_3$&$-e_4$\\
		&$e_2$&$\cdot$&$\cdot$&$e_1$&$\cdot$&$-e_2$&$\cdot$&$-e_5$&$-e_6$\\
		&$e_3$&$\cdot$&$-e_1$&$\cdot$&$\cdot$&$-e_3$&$-e_5$&$\cdot$&$e_7$\\
		&$e_4$&$\cdot$&$-e_2$&$-e_3$&$e_5$&$\cdot$&$\cdot$&$\cdot$&$-e_8$\\ \hline
		$x$&$e_5$&$-e_1$&$\cdot$&$\cdot$&$\cdot$&$e_4$&$-e_6$&$-e_7$&$\cdot$\\
		&$e_6$&$e_2$&$\cdot$&$-e_4$&$-e_6$&$\cdot$&$\cdot$&$-e_8$&$\cdot$\\
		&$e_7$&$-e_3$&$-e_4$&$\cdot$&$-e_7$&$\cdot$&$e_8$&$\cdot$&$\cdot$\\
		&$e_8$&$-e_5$&$e_6$&$-e_7$&$\cdot$&$-e_8$&$\cdot$&$\cdot$&$\cdot$\\
\end{tabular}
\caption{The split para-Cayley algebra multiplication $x\star y$}
\label{tab 1}
\end{table}
}\end{ex}

\begin{rk}\label{rk 234}
{\rm
Let $(A,F,q,\rho, *)$ be a normal twisted composition algebra. Consider the extended algebra $A'=A\otimes_{F_0}F$. We claim this extension algebra $A'$ is also a twisted composition algebra. In fact, we have a nice description of $A'$. Consider an isomorphism of $F$-algebras
\[
\nu: F \otimes_{F_0} F \stackrel{\sim}{\rightarrow} F \times F \times F \quad \text { given by} \quad r_1\otimes r_2\mapsto (r_1r_2, \rho(r_1)r_2, \theta(r_1)r_2).
\]
Note that $\rho\otimes id_F$ is identified with the map defined by $\tilde{\rho}(r_1,r_2,r_3)=(r_2,r_3,r_1)$ to make the following diagram commutative:
\[
\begin{tikzcd}
F\otimes_{F_0} F \arrow[r, "\rho\otimes id_F"] \arrow[d]  & F\otimes_{F_0} F \arrow[d] \\
F\times F\times F \arrow[r, "\tilde{\rho}"]   & F\times F\times F .
\end{tikzcd}
\]
We define the twisted vector spaces $\lx^{\rho}A$ and $\lx^{\theta}A$:
\[
\lx^{\rho}A=\{\lx^{\rho}x \mid x\in A\}, \quad \lx^{\theta}A=\{\lx^{\theta}x \mid x\in A\},
\]
with the operations: $\lx^{\rho}(rx)=\rho(r)\lx^{\rho}x, \lx^{\rho}(x+y)=\lx^{\rho}x+\lx^{\rho}y$, and $\lx^{\theta}(rx)=\theta(r)\lx^{\theta}x, \lx^{\theta}(x+y)=\lx^{\theta}x+\lx^{\theta}y$, for all $x,y\in A, r\in F$. Then there is an $F$-isomorphism $A\otimes_{F_0}F \stackrel{\sim}{\rightarrow} A\times \lx^{\rho}A \times \lx^{\theta}A$ given by:
\[
x\otimes r\mapsto (rx, r (\leftidx^\rho x), r (\leftidx^\theta x))
\]
 (see \cite[Remark 2.3]{KT}). To describe the multiplication in $A\otimes_{F_0}F$ and $A\times \lx^{\rho}A \times \lx^{\theta}A$, we need to consider $F$-bilinear maps:
\[
*_{id}:\lx^{\rho}A \times \lx^{\theta} A \rightarrow A, \quad *_{\rho}:\lx^{\theta} A \times A \rightarrow \lx^{\rho} A, \quad *_{\theta}: A \times \lx^{\rho} A \rightarrow \lx^{\theta} A,
\]
given by
\[
^{\rho}x *_{id} \leftidx^{\theta} y=x * y, \quad ^{\theta} x *_{\rho} y=\lx^{\rho}(x * y), \quad  x *_{\theta}\leftidx^{\rho} y=\lx^{\theta}(x * y),
\]
for all $x,y\in A$. Then the product $\diamond: (A\times \lx^{\rho}A\times \lx^{\theta}A)\times (A\times \lx^{\rho}A\times \lx^{\theta}A) \rightarrow A\times \lx^{\rho}A\times \lx^{\theta}A$ given by 
\[
(x,\lx^{\rho} x,\lx^{\theta} x) \diamond (y,\lx^{\rho} y,\lx^{\theta} y)=(^{\rho} x *_{id} \leftidx^{\theta} y,^{\theta} x *_{\rho} y, x *_{\theta} \leftidx^{\rho} y),
\]
will make the following diagram commutative:
\[
\begin{tikzcd}
(A\otimes_{F_0} F) \times (A\otimes_{F_0} F) \arrow[ r, "*\otimes id_{F}"] \arrow[d]  & A\otimes_{F_0} F \arrow[d] \\
(A\times \lx^{\rho}A\times \lx^{\theta}A)\times (A\times \lx^{\rho}A\times \lx^{\theta}A) \arrow[r, "\diamond"]   & A\times \lx^{\rho}A\times \lx^{\theta}A.
\end{tikzcd}
\] 
Finally, define quadratic forms $^{\rho} q:\lx^{\rho}A \rightarrow F$ and $\lx^{\theta} q: \lx^{\theta}A \rightarrow F$ by 
$$
^\rho q(^\rho x)=\rho(q(x)), \quad ^\theta q(^\theta x)=\theta(q(x)).
$$
We have an isomorphism:
\[
(A\otimes_{F_0} F, F\otimes_{F_0}F, q_F, \rho\otimes id_F, *\otimes id_F) \simeq \left(A \times \lx^{\rho} A \times \lx^{\theta} A, F \times F \times F, q\times^{\rho} q \times^{\theta} q, \tilde{\rho}, \diamond \right).
\]
}
\end{rk}


\section{Special orthogonal groups and triality}

In this section, we recall special orthogonal groups and spin groups for twisted composition algebras. Let $(V,q)$ be a vector space with a nonsingular quadratic form $q$ over a field $F$, with char$(F)$ different from $2$. Denote by $\bb ,\pp$ the bilinear form corresponding to $q$. For any $f\in \End_F(V)$, there exists an element $\sigma_q(f)\in\End_F(V)$ such that $\bb  x,f(y)\pp=\bb  \sigma_b(f)(x),y\pp$. We can see this using matrices: If $b\in \GL(V)$ denotes the Gram matrix of $\bb ,\pp$ with respect to a fixed basis, then $\bb  x,y\pp=x^{t}by$. Let $\sigma_q(f)=b^{-1}f^{t}b$. Then $\bb  x,f(y)\pp=x^tbf(y)=\bb  \sigma_b(f)(x),y\pp$. It is easy to see that $\sigma_q: \End_F(V)\rightarrow \End_F(V)$ given by $f\mapsto \sigma_q(f)$ is an involution of $\End_F(V)$, and we call $\sigma_q$ the involution corresponding to the quadratic form $q$. The special orthogonal group $\SO(V,q)$ is the subgroup of isomorphism group $\Isom(V,q)$ that preserve the form $\bb,\pp$ and have determinant equal to 1: 
\[
\SO(V,q):=\{g\in \Isom(V,q) \mid\bb  g(x),g(y)\pp =\bb  x,y\pp\}.
\]
Elements $g\in \SO(V,q)$ are called proper isometries (Improper isometries are elements in $\Isom(V,q)$ that preserve the form with determinant equal to $-1$).

The Clifford algebra $C(V,q)$ is the quotient of the tensor algebra $T(V):=\oplus_{n\ge 0}V^{\otimes n}$ by the ideal $I(q)$ generated by all the elements of the form $v\otimes v-q(v)\cdot 1$ for $v\in V$. Since $T(V)$ is a graded algebra, $T(V)=T_0(V)\oplus T_1(V)$, where $T_0(V)=T(V\otimes V)$ and $T_1(V)=V\otimes T_0(V)$. This induces a $\ZZ/2\ZZ$-grading of $C(V,q)$:
\[
C(V,q)=C_0(V,q)\oplus C_1(V,q).
\]
We call $C_0(V,q)$ the even Clifford algebra and $C_1(V,q)$ the odd Clifford algebra. When $\dim V=n$, we have $\dim C(V,q)=2^n$, and $\dim C_0(V,q)=2^{n-1}$ (see \cite[Chapter IV]{K1}). For every quadratic space $(V,q)$, the identity map on $V$ extends to an involution on the tensor algebra $T(V)$ which preserve the ideal $I(q)$: $(v_1\otimes \cdots \otimes v_r)^t:=v_r\otimes\cdots\otimes  v_1$ for $v_1,\dots v_r\in V$. It is therefore inducing a canonical involution of the Clifford algebra $\tau: C(V,q)\rightarrow C(V,q)$ given by $\tau(v_1\cdots v_d)=v_d\cdots v_1$. By using the even Clifford algebra $C_0(V,q)$, now we can consider the universal covering of $\SO(V,q)$, which is the spin group $\Spin(V,q)$ defined by:
\[
\Spin (V,q)=\{c\in C_0(V,q)^* \mid c V c^{-1}=V, \tau(c)c=1\}.
\]
For any $c\in \Spin (V,q)$, we have a linear map $\chi_c: x\mapsto cxc^{-1}$. This is an element in $\SO(V,q)$ since $q(\chi_c(x))=cxc^{-1}cxc^{-1}=q(x)$, and we can show that $\Spin(V,q)\rightarrow \SO(V,q)$ given by $c\mapsto \chi_c$ is surjective. We have an exact sequence:
\[
1\rightarrow \ZZ/2\ZZ\rightarrow \Spin(V,q)\rightarrow \SO(V,q)\rightarrow 1.
\]
The special orthogonal group scheme $\SSO(V,q)$ and the spin group scheme $\SSpin(V,q)$ over $F$ are defined by:
\[
\begin{array}{l}
\SSO(V,q)(R):=\{g\in \Isom(V_R,q) \mid\bb  g(x),g(y)\pp =\bb  x,y\pp, \det g=1\},	\\
\SSpin (V,q)(R):=\{c\in C_0(V_R,q)^* \mid c V_R c^{-1}=V_R, \tau(c)c=1\},
\end{array}
\]
for any $F$-algebra $R$, where $V_R:=V\otimes_F R$. In particular, when the quadratic space $(V,q)$ is the normal twisted composition algebra $(V,*)$, we denote by $\SSO(V,*)$ (resp. $\SSpin(V,*)$) the special orthogonal group (resp. spin group) for (split) normal twisted composition algebras.

The Clifford algebra for $(V,*)$ has a special structure. Consider the twisted vector spaces $^\rho V,$ $ ^\theta V$ in Remark \ref{rk 234}. For any $x\in V$, we have two $F$-linear maps 
$$
l_x: \lx^\rho V\rightarrow \lx^\theta V,\quad r_x: \lx^\theta V\rightarrow \lx^\rho V
$$
given by 
\[
l_{x}(\lx^\rho y)=\lx^{\theta}(x * y) \quad \text {and}\quad r_{x}(^{\theta} z)=\lx^{\rho}(z * x).
\]
(see \cite[\S3]{KT}). By Lemma \ref{lm 233}, it follows that the $F$-linear map:
\[
\al: (V,*)\rightarrow \End_{F}(^\rho V \oplus \lx^\theta V),\quad \text{given by}~x\mapsto \left(\begin{array}{cc} {0} & {r_{x}} \\ {l_{x}} & {0}\end{array}\right)
\]
satisfies $\alpha(x)^2=q(x) id$. Hence we can extend this map to:
$\al: C(V,*)\rightarrow \End_{F}(^\rho V \oplus \lx^\theta V)$ by the universal property of Clifford algebra. In fact, this map is an isomorphism of algebras  with involution (see \cite[Proposition (36.16)]{KMRT}):
\[
\alpha: (C(V,*),\tau)\stackrel{\sim}{\rightarrow} (\End_{F}(^\rho V \oplus \lx^\theta V), \sigma_{^\rho q\bot ^\theta q}),
\]
If we restrict this isomorphism to the even Clifford algebra $C_0(V,q)$, we get
 \[
 \alpha: (C_0(V,*),\tau)\stackrel{\sim}{\rightarrow} (\End_{F}({}^\rho V ),\sigma_{^\rho q})\times (\End_{F} ({}^\theta V), \sigma_{^\theta q}),
 \]
where $\sigma_{{}^\rho q}, \sigma_{^\theta q}$ are involutions corresponding to quadratic forms $\lx^\rho q, \lx^\theta q$ (we still use $\al$ to denote the isomorphism for simplicity).

\begin{rk}\label{rk 31}
	{\rm	
		For any $g_1\in\SSO(V,*)(F)$, there exist $g_2, g_3\in\SSO(V,*)(F)$ such hat:
		\[
		g_i(x*y)=g_{i+1}(x)*g_{i+2}(y),\quad i=1,2,3\pmod 3
		\]
		(see \cite[Proposition (36.17)]{KMRT}). It is easy to see that when $g_i$ satisfy the above equation, the following diagram $D(g_i,g_{i+1},g_{i+2})$ commutes:
		\[
		\begin{tikzcd}
			C_0(V,q) \arrow[ r, "\alpha"] \arrow[d, "C_0(g_i)"]  & \End_{F} (\lx^\rho V) \times \End_{F} (\lx^\theta V) \arrow[d, "\Int(\lx^\rho g_{i+1}) \times \Int({}^\theta g_{i+2})"] \\
			C_0(V,q) \arrow[r, "\alpha"]   & \End_{F} (^\rho V) \times \End_{F} (^\theta V).
		\end{tikzcd}
		\]
		Here the automorphism $C_0(g_i): C_0(V,*)\rightarrow C_0(V,*)$ of the even Clifford algebra given by
		\[
		C_0(g_i)(v_1\cdots v_{2r})=g_i(v_1)\cdots g_i(v_{2r}).
		\]
Conversely, for any $g_1,g_2,g_3\in \SSO(V,*)(F)$, if the diagram $D(g_i,g_{i+1},g_{i+2})$ commutes, we will get $g_i(x*y)=g_{i+1}(x)*g_{i+2}(y), i=1,2,3\pmod 3$. These two equivalent statements are called ``the principle of triality".
}\end{rk}

By using this isomorphism $\al$, we have the following result for normal twisted composition algebras. Similar results for symmetric composition algebras can be found in \cite[ Proposition (35.7)]{KMRT}. Let us set $\Spin(V,*):=\SSpin(V,*)(F)$.

\begin{thm}\label{thm 331}
There is an isomorphism 
\[
\Spin(V,*)\cong \{(g_1,g_2,g_3)\in \SO(V,q)^{\times 3} \mid g_i(x*y)=g_{i+1}(x)*g_{i+2}(y),~\text{for any}~ x,y\in V\}
\]
\end{thm}
	
\begin{proof}
Let $c\in C_0(V)^*$. Using the isomorphism with involution $\al$, we obtain $\lx^\rho g_2\in \End_F(^\rho V)$ and $^\theta g_3\in \End_F(^\theta V)$ such that 
\[
\alpha(c)=\left(\begin{array}{cc} {^\rho g_2} & {0} \\ {0} & {^\theta g_3}\end{array}\right)\in \End_{F}(^\rho V )\times \End_{F} (^\theta V).
\]
We have
\[
\alpha(\tau(c)c)=\left(\begin{array}{cc} {^\rho \sigma_q(g_2)} & {0} \\ {0} & {^\theta \sigma_q(g_3)}\end{array}\right)\left(\begin{array}{cc} {^\rho g_2} & {0} \\ {0} & {^\theta g_3}\end{array}\right)=I,
\]
which implies $\sigma_q(g_2)g_2=1, \sigma_q(g_3)g_3=1$, i.e., $g_2, g_3$ are isometries. Consider $\chi_c(x)=cxc^{-1}\in V$. By applying $\alpha$ on both sides, we have $\al(\chi_c(x))=\al(c)\al(x)\al(c^{-1})$, which gives us:
\[
\left(\begin{array}{cc} {0} & {r_{\chi_c(x)}} \\ {l_{\chi_c(x)}} & {0}\end{array}\right)=\left(\begin{array}{cc} {0} & {\lx^\rho g_2\cdot r_x\cdot \lx^\theta \sigma_q(g_3)} \\ {^\theta g_3\cdot l_x\cdot \lx^\rho \sigma_q(g_2)} & {0}\end{array}\right).
\]
	This is equivalent to $^\rho g_2 r_x=r_{\chi_c(x)}\leftidx^\theta g_3$, $^\theta g_3 l_x=l_{\chi_c(x)}\leftidx^\rho g_2$, i.e.,
\begin{equation}\label{eq 33}
g_2(x*y)=g_3(x)*\chi_c(y),\quad g_3(x*y)=\chi_c(x)*g_2(y).
\end{equation}
Finally, $\chi_c$ is an isometry since $q(\chi_c(x))=cxc^{-1}cxc^{-1}=q(x)$. By Remark \ref{rk 31}, Equation (\ref{eq 33}) yields the diagram $D(\chi_c, g_2, g_3)$ commuting, which shows that $C_0(\chi_c)$ is the identity on the center of $C_0(V,*)$. By \cite[Proposition (13.2)]{KMRT}, the isometry $\chi_c$ is proper. Similarly, $g_2, g_3$ are also proper isometries. Thus, let $g_1=\chi_c$. We get related equations as above. We now send $c\mapsto (g_1,g_2,g_3)$ that gives as above. This giving map is an injective group homomorphism since $\al$ is an isomorphism. It is also surjective, since, given any $(g_1,g_2,g_3)$ satisfying $g_i(x*y)=g_{i+1}(x)*g_{i+2}(y)$, there exist $c\in C_0(V)$ such that $\alpha(c)=\left(\begin{array}{cc} {^\rho g_2} & {0} \\ {0} & {^\theta g_3}\end{array}\right)$.
\end{proof}

From Theorem \ref{thm 331}, we have an isomorphism between group schemes over $F_0$:
\[
\Res_{F/F_0}(\SSpin(V,*))(R)\cong \{(g_1,g_2,g_3)\in \Res_{F/F_0}(\SSO(V,*)(R)^{\times 3} \mid g_i(x*y)=g_{i+1}(x)*g_{i+2}(y) \}
\]
for any $F_0$-algebra $R$. The transformation $\tilde{\rho}: (g_1,g_2,g_3)\mapsto (g_2,g_3,g_1)$ is an outer automorphism of $\Res_{F/F_0}(\SSpin(V,*))$ satisfying $\tilde{\rho}^3=1$. Here $\tilde{\rho}$ generates a subgroup of $\Aut(\Res_{F/F_0}(\SSpin(V,*)))$, which is isomorphic to $A_3$. Consider the fixed points of $\Res_{F/F_0}(\SSpin(V,*))$ under $A_3=\bb \tilde{\rho}\pp$. We obtain the triality group for the special orthogonal groups $G$:
\begin{flalign*}
	G(R) &:=\Res_{F/F_0}(\SSpin(V,*))^{A_3}(R)\\
	  &\cong \{ g\in \SSO(V,*)(R\otimes_{F_0}F) \mid g(x*y)=g(x)*g(y) ~\text{for all}~x,y\in V\otimes_{F_0}R \}.
\end{flalign*}
for any $F_0$-algebra $R$.


\section{Affine Grassmannians for triality groups}

In this section we give the definition of affine Grassmannians for triality groups. Recall that the affine Grassmannian for general groups is representable by an ind-scheme and is a quotient of loop groups. 

Let $k$ be a field. We consider the field $K=k\llp t\rlp$  of Laurent power series with indeterminate $t$ and coefficients in $k$. Let $\calO_K=k\lp t\rp$ be the discretely valued ring of power series with coefficients in $k$. For a $k$-algebra $R$, we set $\DD_R=\Spec(R\lp t\rp)$, resp. $\DD_R^*=\DD_R\setminus \{t=0\}=\Spec(R\llp t\rlp)$, which we picture as an $R$-family of discs, resp., an $R$-family of punctured discs.

We recall some functors from \cite[\S 1]{PR3}: Let $X$ be a scheme over $K$. We consider the functor $LX$ from the category of $k$-algebras to that of sets given by 
\[
R\mapsto LX(R):= X(R\llp t\rlp).
\]
If $\calX$ is a scheme over $\calO_K$, we denote by $L^+\calX$ the functor  from the category of $k$-algebras to that of sets given by 
\[
R\mapsto L^+\calX(R):= \calX(R\lp t\rp).
\]
The functors $LX, L^+\calX$ give sheaves of sets for the fpqc topology on $k$-algebras. In what follows, we will call such functors ``$k$-spaces" for simplicity. 

\begin{ex}
{\rm 	
If $\calX={\bf A}^r_{\calO_K}$ is the affine space of dimension $r$ over $\calO_K$, then $L^+\calX$ is the infinite dimensional affine space $L^+\calX=\prod\limits_{i=0}^{\infty}{\bf A}^r$, via:
\[
L^+\calX(R)=\Hom_{k\lp t\rp}(k\lp t\rp[T_1,...,T_r],R\lp t\rp)=R\lp t\rp^r=\prod_{i=0}^{\infty}R^r=\prod\limits_{i=0}^{\infty}{\bf A}^r(R).
\]
Let $\calX$ be the closed subscheme of ${\bf A}^r_{\calO_K}$ defined by the vanishing of polynomials $f_1,...,f_n$ in $k\lp t\rp[T_1,...,T_r]$. Then $L^+\calX(R)$ is the subset of $L^+{\bf A}^r(R)$ of $k\lp t\rp$-algebra homomorphisms $k\lp t\rp[T_1,...,T_r]\rightarrow R\lp t\rp$ which factor through $k\lp t\rp[T_1,...,T_r]/(f_1,...,f_n)$. If $X$ is an affine $K$-scheme, $LX$ is represented by a strict ind-scheme.
}\end{ex}

Let $X$ be a linear algebraic group over $K$. The loop group associated to $X$ is the ind-scheme $LX$ over $\Spec (k)$. We list some properties of loop groups:
\begin{itemize}
	\item [(1)] $L(X\times_{k}Y)=LX\times_{k} LY$;
	\item [(2)] If $k^{\prime}$ is a $k$ -field extension, then we have an isomorphism of ind-schemes over $k^{\prime}$
\[
LX \times_{k} \Spec(k^{\prime}) \simeq L(X\times_{k\llp t\rlp} \Spec(k^{\prime}\llp t\rlp))
\]
\item[(3)] Assume that $K^{\prime} / K$ is a finite extension of $K$, where $K^{\prime}=k\llp u\rlp$. If $X=\Res_{K^{\prime} / K} H$ for some linear algebraic group $H$ over $K^{\prime},$ then we have an isomorphism of ind-schemes over $k$:
\[
L X \simeq L H,
\]
by
\[
(L X)(R)=X(R\llp t\rlp)=H(R\llp t\rlp\otimes_{k\llp t\rlp} k\llp u\rlp)=H((R\llp u\rlp)=LH(R).
\]
\end{itemize}

Now let $\calX$ be a flat affine group scheme of finite type over $k\lp t\rp$. Let $X=\calX_{\eta}$ denote the generic fiber of $\calX$. This is a group scheme over $k\llp t\rlp$.  We consider the quotient sheaf over $\Spec(k)$:
\[
\calF_{\calX}:=L \calX_{\eta} / L^{+} \calX.
\]
This is the fpqc sheaf associated to the presheaf which to a $k$ -algebra $R$ associates the quotient $\calX(R\llp t\rlp) /\calX(R \lp t \rp)$. Generally, the  affine Grassmannian for $\calX$ is the functor $\Gr_\calX:\Alg_k\rightarrow\Sets$ 	which associates to a $k$-algebra $R$ the isomorphism classes of pairs $(\calE,\al)$ where $\calE\rightarrow \DD_R$ is a left fppf $\calX$-torsor and $\al\in\calE(\DD_R^*)$ is a section.

Here a pair $(\mathcal{E}, \alpha)$ is isomorphic to $(\mathcal{E}^{\prime}, \alpha^{\prime})$ if there exists a morphism of $\calX$ -torsors $\pi: \mathcal{E} \rightarrow \mathcal{E}^{\prime}$ such that $\pi \circ \alpha=\alpha^{\prime}$. The datum of a section $\alpha \in \mathcal{E}(\DD_{R}^{*})$ is equivalent to the datum of an isomorphism of $\calX$ -torsors
\[
\mathcal{E}_{0}|_{\DD_{R}^{*}} \stackrel{\simeq}{\longrightarrow} \mathcal{E} |_{\DD_{R}^{*}}, \quad g \mapsto g \cdot \alpha,
\]
where $\mathcal{E}_{0}:=\calX$ is viewed as the trivial $\calX$ -torsor. The loop group $L	X$ acts on the affine Grassmannian via $g\cdot [(\calE,\al)]=[(\calE, g\al)]$.

\begin{prop}\label{prop 422}
If $\calX\rightarrow \Spec(k\lp t\rp)$ is a smooth affine group scheme, then the map $LX\rightarrow \Gr_\calX$ given by $g\mapsto [(\calE_0,g)]$ induces an isomorphism of fpqc quotients:
\[
\calF_\calX\cong \Gr_\calX.
\]
\end{prop}

\begin{proof}
See \cite[Proposition 1.3.6]{zhu}.
\end{proof}

Here are a few observations:
\begin{itemize}
	\item [(1)] If $\rho: \calX\rightarrow H$ is a map of group schemes which are flat of finite presentation over $k\lp t\rp$, then there is a map of functors:
\[
\Gr_\calX\rightarrow \Gr_H,\quad (\calE,\al)\mapsto (\rho_* \calE, \rho_*\al),
\]
where $\rho_* \calE=H\times^{\calX}\calE$ denotes the push out of torsors, and $\rho_*\al=(id,\al):(H\times^\calX\calE_0) |_{\DD_{R}^{*}} \rightarrow (H\times^\calX\calE) |_{\DD_{R}^{*}}$ in this description. 
	\item [(2)] If $k'$ is a $k$-field extension, then we have:
\[
\Gr_\calX\times_{k}\Spec(k')\simeq \Gr_{\calX\times_{k\lp t\rp}\Spec(k'\lp t\rp)}.
\]
\end{itemize}

\begin{rk}{\rm 
When $X=\GL_{n}$, a $G$ -bundle on $\mathcal{E} \rightarrow \DD_{R}$ is canonically given by a rank $n$ vector bundle, i.e., a rank $n$ locally free $R \lp t\rp$-module $L$. The trivialization $\alpha$ induces an isomorphism of $R\llp t\rlp$-modules $L[t^{-1}]\simeq R\llp t\rlp^n$. By taking the image of $L \subset L\lp t\rp$ under this isomorphism, we obtain a well defined finite locally free $R\lp t\rp$ -module $\Lambda=\Lambda_{(\mathcal{E}, \alpha)} \subset R\llp t\rlp^n$ such that $\Lambda[t^{-1}]=R\llp t\rlp^n$. Note that $\Lambda$ depends only on the class of $(\mathcal{E}, \alpha)$. 	
}\end{rk}

Now we can define affine Grassmannians for triality groups. In what follows, let $k$ be a field with char$(k)\neq 2,3$. Suppose that the cubic primitive root $\xi$ is in $k$. We set $F=k\llp u\rlp, F_0=k\llp t\rlp$ with $u^3=t$. Thus $F/F_0=k\llp u\rlp/k\llp t\rlp$ is a cubic Galois field extension. Set $\Gamma=\Gal(F/F_0)$ with generator $\rho$ with $\rho(u)=\xi u$. Then $k\lp t\rp$ (resp. $k\lp u\rp$) is the ring of integers of $F_0$ (resp. $F$). 

Recall that $(V,*)$ is a normal twisted composition algebra obtained from the para-Cayley algebra over $F$, i.e., there is a basis $\{e_1,...,e_8\}$ of $(V,*)$ in the Table \ref{tab 1}, with the multiplication 
\[
x*y=(id_C\otimes \rho) (x)\star (id_C\otimes \theta)(y) \quad \text{for all}~x,y\in C\otimes_{F_0} F,
\]
where $(C,\star)$ is the split para-Cayley algebra. The quadratic form of $(V,*)$ is determined by the multiplication from Lemma \ref{lm 233}. Denote by $\bb ,\pp$ the bilinear form: $\bb  ,\pp : V\otimes V\rightarrow F$ corresponding to the quadratic form. Let $R$ be an $F_0$-algebra. Notice that the base change $V\otimes_{F_0}R$ is isomorphic to $R\llp u\rlp^8$. A finitely generated projective submodule in $V\otimes_{F_0}R$ is called a lattice in $V\otimes_{F_0}R$. We set $\LL=\oplus_{i=1}^{8}R\lp u\rp e_i$, and call it the standard lattice in $V\otimes_{F_0}R$. 

In \S3, we defined the triality group for special orthogonal groups over $F_0=k\llp t\rlp$:
\begin{flalign*}
	G(R) &=\Res_{F/F_0}(\SSpin(V,*))^{A_3}(R)\\
	&\cong \{ g\in \SSO(V,q)(R\otimes_{k\llp t\rlp}k\llp u\rlp) \mid g(x*y)=g(x)*g(y) ~\text{for all}~x,y\in V\otimes_{k\llp t\rlp}k\llp u\rlp \}.
\end{flalign*}
for any $k\llp t\rlp$-algebra $R$. Let $\mathscr{G}$ be the affine group scheme over $k\lp t\rp $ that represents the functor from $k\lp t\rp$-algebras to groups that sends $R$ to 
\[
\scrG(R):=\{g\in \SSO_8(k\lp u\rp \otimes_{k\lp t\rp }R) \mid g(x*y)=g(x)*g(y) ~\text{for all}~x,y\in \LL\}.
\]
Here $\scrG$ is the parahoric subgroup given by $\LL$ by Proposition 1.3.9, \cite{PK}. We denote by $L\scrG_{\eta}$ (resp. $L^+\mathscr{G}$) the functor from the category of $k$-algebras to groups given by $L\scrG_{\eta}(R)=\scrG_{\eta}(R\llp t\rlp)$ (resp. $L^+\mathscr{G}(R)=\mathscr{G}(R\lp t\rp)$). The quotient fpqc sheaf $L\scrG_{\eta}/L^+\mathscr{G}$ is by definition the affine Grassmannian for the triality group. Our main theorem  is:
\begin{thm}\label{thm 511}
There is an $L\scrG_{\eta}$-equivariant isomorphism 
\[
L\scrG_{\eta}/L^+\mathscr{G}\simeq \mathscr{F}
\]	
where the functor $\mathscr{F}$ sends a $k$-algebra $R$ to the set of finitely generated projective $R\lp u\rp $-modules $L$ (i.e., $R\lp u\rp $-lattices) of $V\otimes_{k}R\cong R\llp u\rlp^8$ such that
\begin{itemize}
	\item [(1)] $L$ is self dual under the bilinear form $\bb\ ,\ \pp$, i.e., $L\simeq \Hom_{R\lp u\rp}(L, R\lp u\rp )$.
	\item [(2)] $L$ is closed under multiplication, $L*L\subset L$.
	\item [(3)]  There exists $a\in L$, such that $q(a)=0$,  $\bb  a*a,a\pp=1$.

	\item [(4)] For $a$ as in (3), let $e=a+a*a$. Then, we have $\overline{e*x}=-\bar{x}=\overline{x*e}$ for any $\bar{x}$ satisfying $\bb  \bar{x},\bar{e}\pp =0$. (Here, $\bar{x}$ is the image of $x$ under the canonical map $L\rightarrow L/uL$.)\end{itemize}
\end{thm}

This theorem is proven in the next section. It gives a bijection between $k$-points in the affine Grassmannian for the triality group and a certain set of $k\lp u\rp$-lattices in $V$ that satisfy some special conditions.


\section{Proof of the main theorem}

Similar to the proof of \cite[Theorem 4.1]{PR3} in the unitary group case, it suffices to check the following two statements:
\begin{itemize}
	\item [(i)] For any $R$, $g\in L\scrG_{\eta}(R)$, $L=g(\LL)$ satisfies condition (1)-(4).
	\item [(ii)] For any $L\in \mathscr{F}(R)$ with $(R,\calM)$ a local henselian ring with maximal ideal $\calM$, there exists $g\in  L\scrG_{\eta}(R)$ such that $L=g(\LL)$.
\end{itemize}

Part (i) is easy to prove, since $g$ preserves the bilinear form $\bb ,\pp$ and the product $*$. For any $x,y\in L$, let $x=g(x_0), y=g(y_0)$ where $x_0,y_0\in\LL$. Then $x*y=g(x_0)*g(y_0)=g(x_0*y_0)\in L$, so (2) satisfied. (1) is obvious via $\bb  g(x),g(y)\pp=\bb  x,y\pp$. For (3), let $a=g(e_4)$. Then $\bb  a*a,a\pp=\bb  g(e_4)*g(e_4),g(e_4)\pp =\bb  g(e_5),g(e_4)\pp=\bb  e_4,e_5\pp=1$, and $q(a)=q(e_4)=0$. For $g(e)=g(a)+g(a*a)$, we have $g(e)*g(x)+g(x)=g(x)*g(e)+g(x)=0$ for any $g(x)$ satisfying $\bb  g(x),g(e)\pp =\bb  x,e\pp =0$. 

To prove part (ii), the key is to find a basis in $L$ such that the multiplication table under the basis is the same as Table \ref{tab 1}, i.e., we need to find a basis $\{f_i\}\in L$ such that $f_i*f_j=f_k$ for $e_i*e_j=e_k$ in the Table \ref{tab 1}. Thus we can define $g$ by $g(e_i)=f_i$, and $g$ is then in $L\scrG_{\eta}(R)$.

We claim that $a$ as in assumption (3) is a primitive element in $L$ (an element in $L$ that extends to a basis of $L$). Consider the quotient map
\[
R\lp u\rp \rightarrow R\rightarrow R/\calM=\kappa,
\]
where $\kappa$ is the residue field of $R$. There is a base change $L\rightarrow L\otimes_{R\lp u\rp }\kappa$, and we still denote by $\bar{x}$ the image of $x\in L$. Consider $\bar{a}\in  L\otimes_{R\lp u\rp }\kappa$. We have $\bb \overline{a*a},\bar{a}\pp=1$, hence $\bar{a}\neq 0$. By Nakayama's lemma, we can extend $a$ to a basis of $L$. Similarly, we can show that $a*a$ is also a primitive element. Here $a, a*a$ are independent by $\bb  a,a*a\pp =1$. Let $v_1,...,v_6$ be any base extension for $a,a*a$. We define a sublattice $L_0\subset L$: 
\[
L_0:=\{x\in L \mid \bb  x,a\pp =0, \bb  x,a*a\pp =0\}.
\]
For any $x\in L$, we can write $x$ as $\sum_{i=1}^6r_iv_i+r_7 a+r_8(a*a)$ for some $r_i\in R\lp u\rp$. Consider $v'_i=v_i-\bb  a,v_i\pp a*a-\bb  a*a,v_i\pp a$. It is easy to see that $\bb  v_i',a\pp =0, \bb  v_i',a*a\pp =0$, so $v_i'\in L_0$. And $v_i', a, a*a$ are linear independent. We obtain
\[
x=\sum_{i=1}^6r_iv'_i+(r_7+\sum_{i=1}^6 r_i\bb  v_i,a*a\pp) a+(r_8+\sum_{i=1}^6 r_i\bb  v_i,a\pp)(a*a).
\]
Therefore, $L=R\lp u\rp  a\oplus R\lp u\rp (a*a)\oplus L_0$, where $L_0$ is a sublattice of rank 6.

Set $f_1=a, f_2=a*a$. Here $f_1, f_2$ play similar roles as for $e_4,e_5$ in the Table \ref{tab 1}. By Lemma \ref{lm 232} and Lemma \ref{lm 233}, we obtain a hyperbolic subspace $R\lp u\rp  a\oplus R\lp u\rp (a*a)$ with:
\[
\begin{array}{l}
f_1*f_1=f_2, \quad f_2*f_2=f_1,\\
f_1*f_2=f_2*f_1=0,\\
q(f_1)=q(f_2)=0, \quad \bb  f_1,f_2\pp =1.	
\end{array}
\]

\begin{lm}\label{lm512}
We have
\[
L_0*f_i\subset L_0,\quad f_i*L_0\subset L_0,
\]
for $i=1,2$.
\end{lm}

\begin{proof}
For any $x\in L_0$, we have $\bb  x*f_i,f_i\pp =\rho(\bb  f_i*f_i,x\pp )=0$, and $\bb  x*f_i, f_{i+1}\pp =\rho(\bb  f_i*f_{i+1},x\pp)=0$ by Lemma \ref{lm 233}. Similarly for $f_i*x$.
\end{proof}

Define the $\rho$-linear transformations $t_i: L_0\rightarrow L_0$, given by $t_i(x)=x*f_i$ for $i=1,2$. Take $L_i=t_i(L_0)=L_0*f_i$. Trivially, $t_i(L_i)\subset L_i$. Both $L_i$ are isotropic with rank $(L_i) \leq 3$ since $f_i$ is an isotropic element. For any $x\in L_0$, we have
\[
(f_2*x)*f_1+(f_1*x)*f_2=\theta(\bb  f_1,f_2\pp )x=x,
\]
by Lemma \ref{lm 233}. So $L_0=L_1+L_2$. Since rank$(L_i)\leq 3$, we must have a direct sum composition: $L_0=L_1\oplus L_2$. 

\begin{lm}\label{lm 513}\
\begin{itemize}
	\item [(1)] For any $x\in L_0$, $t_i^2(x)=- f_{i+1}*x$ $(i=1,2 \mod 2)$.
	\item [(2)] For any $x\in L_i$, $t_i^{3}(x)=-x$.
	\item [(3)] From (2), $t_i$ is a $R\lp t\rp$-isomorphism when restricted at $L_i$, more precisely, we have $t_i: L_i\rightarrow L_i,  ~x\mapsto x*f_i$. The inverse map $t_i^{-1}=-t_i^2$ is a $\theta$-linear transformation.
	\item [(4)] For $x\in L_1, y\in L_2$, we have $\bb  t_1(x),t_2(y)\pp =\rho(\bb  x,y\pp )$.
\end{itemize}
\end{lm}

\begin{proof}
(1) For $x\in L_0$, we have $t_1^2(x)=((x*f_1)*f_1)=-((f_1*f_1)*x)=-(f_2*x)$ by Lemma \ref{lm 233}. A similar argument gives $t_2^2(x)=-f_1*x$.

(2) For any $x\in L_1$, we have $t_1^3(x)=-((f_2*x)*f_1)$. Consider
\[
(f_2*x)*f_1+(f_1*x)*f_2=\theta(\bb  f_1,f_2\pp )x=x,
\]
by Lemma \ref{lm 233}. Let $x=z*f_1\in L_1$ for some $z\in L_0$. Then $f_1*x=f_1*(z*f_1)=0$ by $q(f_1)=0$. Hence $(f_2*x)*f_1=x$, and we obtain $t_1^3(x)=-x$. Similar calculations for $y\in L_2$, and gives $t_2^3(y)=-y$.

Part (3) follows from (2) immediately. For (4), we know that $\bb  t_1(x),t_2(y)\pp =\bb  x*f_1,y*f_2\pp =\rho(\bb  f_1*(y*f_2),x\pp )$, and 
\[
f_1*(y*f_2)=-t_2^2(y*f_2)=-t_2^2\cdot t_2(y)=-t_2^3(y)=y,
\]
by (1) and (2). Hence $\bb  t_1(x),t_2(y)\pp =\rho(\bb  x,y\pp )$.
\end{proof}

\begin{rk}
\rm{
(1) From the proof of above Lemma, we can see that $f_i*L_i=0$, and $L_i*f_{i+1}=0$ for $i=1,2 \mod 2$.

(2) Since $L_1, L_2$ are isotropic and $\bb  ,\pp $ restricted to $L_0$ is nondegenerate, the $L_i$ are in duality by the isomorphism $L_1\rightarrow L_2^{\vee}$ given by $x\mapsto \bb  x,-\pp$. Hence $L_1\simeq \Hom (L_2, R\lp u\rp )$. 
}
\end{rk}	

\begin{lm}\label{lm 515}
We have
\begin{itemize}
	\item [(1)] $L_1*L_2\subset R\lp u\rp f_1, \quad L_2*L_1 \subset R\lp u\rp f_2$,
	\item [(2)] $L_i*L_i\subset L_{i+1}$ $(i=1,2 \mod 2)$.
\end{itemize}	
\end{lm}

\begin{proof} (1) For any $x\in L_1, y\in L_2$, we write $x$ as $x=x_1*f_1$ with $x_1\in L_1$, and $y$ as $y=y_1*f_2$ with $y_1\in L_2$. Consider 
\[
x*y=(x_1*f_1)*(y_1*f_2)=-((y_1*f_2)*f_1)*x_1+\theta(\bb  x_1,y_1*f_2\pp )f_1,
\]
by Lemma \ref{lm 233}. Notice that $(y_1*f_2)*f_1\in L_2*f_1=0$. Thus we have $x*y=\theta(\bb  x_1,y_1*f_2\pp )f_1$. Further,
\begin{flalign*}
	\bb  x_1,y_1*f_2\pp &=\theta(\bb  t_1(x_1),t_2(y_2*f_2)\pp )\\
	&=\theta(\bb  x,t_2(y)\pp),
\end{flalign*}
by Lemma \ref{lm 513} (4). Hence $x*y=\rho(\bb  x,t_2(y)\pp )f_1$. Similarly, we have $y*x=\rho(\bb  t_1(x),y\pp )f_2$.

(2) For any $x_1,x_2\in L_1$, we first claim that $x_1*x_2\in L_0$. Consider $\bb  x_1*x_2,f_1\pp =\theta(\bb  f_1*x_1,x_2\pp )=0$ by $f_1*L_1=0$, and $\bb  x_1*x_2,f_2\pp =\rho(\bb  x_2*f_2,x_1\pp )=0$ by $L_1*f_2=0$. Using Lemma \ref{lm 233}, we find that
\begin{flalign*}
	t_1(x_1)*t_1(x_2)&=(x_1*f_1)*(x_2*f_1)\\
	&=-f_1*(x_2*(x_1*f_1))\\
	&=f_1*(f_1*(x_1*x_2))
	\end{flalign*}
by $\bb  x_1*f_1,f_1\pp =0$ and $\bb  f_1,x_2\pp =0$. We also have $f_1*(f_1*(x_1*x_2))=f_1*(-t_2^2(x_1*x_2))=t_2^4(x_1*x_2)=-t_2(x_1*x_2)$. Therefore, 
\[
t_1(x_1)*t_1(x_2)=-t_2(x_1*x_2).
\]
Since $x_1*x_2\in L_0$, we obtain that $t_2(x_1*x_2)\in L_2$. Hence $L_1*L_1\subset L_2$. Similarly, $L_2*L_2\subset L_1$.
\end{proof}

We now prove that $L$ has the same multiplication table as the Table \ref{tab 1}: We want to find a basis $\{x_1,x_2,x_3\}$ of $ L_1$ (resp. $\{y_1,y_2,y_3\}$ of $ L_2$) such that $t_1(x_i)=-id$ (resp. $t_2(y_i)=-id$). Consider the quotient map $R\lp u\rp\rightarrow \kappa=R\lp u\rp/(\calM,u)$. We set $\bar{L}=L\otimes_{R\lp u\rp}\kappa$, $\bar{L}_i=L_i\otimes_{R\lp u\rp}\kappa$ with multiplication $\bar{x}\star\bar{y}=\overline{x*y}$, and 
\[
\bar{t}_i: \bar{L}_i\rightarrow \bar{L}_i, \quad\text{given by}\quad \bar{t}_i(\bar{x})=\bar{x}\star \bar{f_i},
\] 
for $i=0,1,2$.

\begin{prop}\label{prop 6}
Given $(L,*,\bb  ,\pp )$ satisfying (1)-(4) as above. Then $(\bar{L},\star)$ is isomorphic to the split para-Cayley algebra.
\end{prop}

\begin{proof}
It is easy to see $q(\bar{x}\star\bar{y})=q(\bar{x})q(\bar{y})$, and $\bb  \bar{x}\star\bar{y},\bar{z}\pp =\bb  \bar{y}\star \bar{z},\bar{x}\pp $, so $\bar{L}$ is a symmetric composition algebra. By \cite[Lemma (34.8)]{KMRT}, a symmetric algebra is a para-Cayley algebra if and only if it admits a para-unit, i.e., there exist an element $\bar{e}\in \bar{L}$, such that 
\[
\bar{e}\star\bar{e}=\bar{e},\quad \bar{e}\star\bar{x}=\bar{x}\star\bar{e}=-\bar{x},
\]
for all $\bar{x}\in \bar{L}$ satisfying $\bb  \bar{e},\bar{x}\pp =0$. Set $e=f_1+f_2$ in our case. We can see that $e$ is an idempotent element by 
$e\star e=(f_1+f_2)*(f_1+f_2)=f_1+f_2=e$. By condition (4), we get $\bar{e}\star\bar{x}=\bar{x}\star\bar{e}=-\bar{x}$, for all $\bar{x}\in \bar{L}$ satisfying $\bb  \bar{e},\bar{x}\pp =0$. Thus $\bar{e}$ is a para-unit in $\bar{L}$, and $\bar{L}$ is a para-Cayley algebra. It is split since $q$ is an isotropic norm.	
\end{proof}

\begin{lm}\label{lm 517}
For $\bar{t}_i: \bar{L}_i\rightarrow \bar{L}_i$, we have $\bar{t}_i(\bar{x})=\bar{x}\star \bar{f_i}=-\bar{x}$ for any $\bar{x}\in \bar{L}_i$. Then $\bar{L}_i=\bar{L}_0\star \bar{f}_i=\{\bar{x}\in \bar{L}_0 \mid \bar{x}\star\bar{f}_i=-\bar{x}\}$, $i=1,2$.
\end{lm}

\begin{proof}
By \cite[Lemma (34.8)]{KMRT}, we can define $\bar{x}\diamond \bar{y}=(\bar{e}\star\bar{x})\star(\bar{y}\star\bar{e})$ as a unital composition algebra with the identity element $\bar{e}$. We have $\bar{x}\star\bar{y}=r(\bar{x})\diamond r(\bar{y})$, where $r(\bar{x})=\bb  \bar{e},\bar{x}\pp \bar{e}-\bar{x}$ is the conjugation of $\bar{x}$. By \cite[Proposition 1.2.3]{Sp},
\[
\bar{x}\diamond\bar{y}+\bar{y}\diamond \bar{x}-\bb  \bar{x},\bar{e}\pp \bar{y}-\bb  \bar{y},\bar{e}\pp \bar{x}+\bb  \bar{x},\bar{y}\pp \bar{e}=0.
\]
Using $\bar{x}\star\bar{y}=r(\bar{x})\diamond r(\bar{y})$ and $\bb  r(\bar{x}),r(\bar{y})\pp =\bb  \bar{x},\bar{y}\pp $, we obtain
\[
\bar{x}\star\bar{y}+\bar{y}\star\bar{x}=\bb  \bar{e},\bar{x}\pp r(\bar{y})+\bb  \bar{e},\bar{y}\pp r(\bar{x})-\bb  \bar{x},\bar{y}\pp \bar{e}.
\]
Let $\bar{y}=\bar{f}_i$. We get
$\bar{x}\star\bar{f}_i+\bar{f}_i\star\bar{x}=r(\bar{x})$. Therefore, if $\bar{x}\in \bar{L}_0*\bar{f}_i$, we have $\bar{f}_i\star\bar{x}=0$ by $q(\bar{f}_i)=0$, and
\[
\bar{x}\star\bar{f}_i=\bar{x}\star\bar{f}_i+\bar{f}_i\star\bar{x}=\bb  \bar{e},\bar{x}\pp \bar{e}-\bar{x}=-\bar{x}.
\]
This implies $\bar{L}_0\star\bar{f}_i\subset \{\bar{x}\in \bar{L}_0 \mid \bar{x}\star\bar{f}_i=-\bar{x}\}$. It is obvious that $\{\bar{x}\in \bar{L}_0 \mid \bar{x}\star\bar{f}_i=-\bar{x}\}\subset \bar{L}_0\star\bar{f}_i$. Hence we get 
\[
\bar{L}_i=\bar{L}_0\star\bar{f}_i=\{\bar{x}\in \bar{L}_0 \mid \bar{x}\star\bar{f}_i=-\bar{x}\},
\]
and $\bar{t}_i=-id$.
\end{proof}

So far we know $t_i: L_i\rightarrow L_i$ is a $\rho$-linear isomorphism with $t_i^3=-id$, and $\bar{t}_i=-id$. We will use non-abelian Galois cohomology to prove that $t_i$ and $-id$ are the same up to $\rho$-conjugacy. More precisely, if we fix a basis for $L_i\cong R\lp u\rp ^3$ and let $A_i\in \GL_3(R\lp u\rp )$ represent $t_i$, we can find a new basis for $L_i$ with transition matrix $b\in \GL_3(R\lp u\rp )$, such that 
\[
-I=b^{-1}A_i \rho(b).
\]
Let $\Gamma=\{1,\rho,\theta\}$ be the cyclic group. Set $B=\Aut(L_1)=\GL_3(R\lp u\rp )$. Consider the quotient map $R\lp u\rp \rightarrow \kappa$. Since $(R\lp u\rp , (u)), (R,\calM)$ are henselian pairs, we obtain the exact sequence:
\[
1\rightarrow U\rightarrow \GL_3(R\lp u\rp )\rightarrow \GL_3(\kappa)\rightarrow 1
\]
where $U$ is the kernel of $\GL_3(R\lp u\rp )\rightarrow \GL_3(\kappa)$. Here $\Gamma$ acts on $\GL_3(R\lp u\rp )$ by $\rho(u)=u\xi$, and $\Gamma$ acts trivially on $\GL_3(\kappa)$. We obtain the exact sequence of pointed sets:
 \[
 1\rightarrow U^\Gamma\rightarrow \GL_3(R\lp u\rp )^\Gamma \rightarrow \GL_3(\kappa)^\Gamma \rightarrow H^1(\Gamma,U)\rightarrow H^1(\Gamma,\GL_3(R\lp u\rp )\rightarrow H^1(\Gamma,\GL_3(\kappa)).
 \]
by \cite[Proposition 38]{Se}. Since $U$ is a unipotent group over $k\lp u\rp$ with char$(k)\neq 3$, we have $H^1(\Gamma, U)=1$. Hence the only element mapped to the base point of $H^1(\Gamma,\GL_3(k))$ is the base point of $H^1(\Gamma,\GL_3(R\lp u\rp )$, i.e., for any $[a_s]\in H^1(\Gamma,\GL_3(R\lp u\rp ))$ satisfying $[\bar{a}_s]= 1$, we have $[a_s]=1$.

Consider $t_1: L_1\rightarrow L_1$. The subgroup of $\GL_3(R\lp u\rp )$ generated by $t_1$ is $\{1,t_1^2,-id,-t_1,-t_1^2,id\}$ given by $t_1^3=-id$. If we fix the basis and use $A_1$ to represent $t_1$, we get $t_1^2=A_1\rho(A_1)$, $t_1^3=A_1\rho(A_1)\theta(A_1)=-I$. Define a map:
$$
a:\Gamma\rightarrow \GL_3(R\lp u\rp)
$$
given by $\rho\mapsto a_{\rho}=-A_1$. Using $a_{st}=a_s\leftidx^s a_t$, we get $\theta\mapsto a_{\theta}=a_{\rho}\rho(a_{\rho})=A_1\rho(A_1)$, and $1\mapsto a_{1}=I$. Hence the image of  $\Gamma=\{\rho,\theta, 1\}$ is the subgroup $\{t_1^4=-t_1,t_1^8=t_1^2, t_1^{12}=id\}\subset \bb  t\pp$. This is a 1-cocycle. Consider the image $[\bar{a}]$ of $[a]$ under the map
$$
1\rightarrow H^1(\Gamma,\GL_3(R\lp u\rp )\rightarrow H^1(\Gamma,\GL_3(k)).
$$
We get $[\bar{a}_{\rho}]=-[\bar{t}]=1$ by Lemma \ref{lm 517}. Therefore $[a_{\rho}]=1$. In matrix language, there exist $b\in \GL_3(R\lp u\rp )$ such that
$$
I=b^{-1}(-A_1)\rho(b), \quad t_1\sim -id.
$$
We have a similar conclusion for $t_2$.

Using the above we see that there exist a basis $\{x_1,x_2,x_3\}$ for $ L_1$, and a dual basis $\{y_1,y_2,y_3\}$ for $ L_2$, such that $t_1(x_i)=-x_i$, $t_2(y_i)=-y_i$, $\bb  x_i,y_j\pp =\delta_{ij}$. By Lemma \ref{lm 513}, we have 
$$
\begin{array}{l}
	x_i*f_1=-x_i, \quad f_1*x_i=0,\\
	x_i*f_2=0, \quad f_2*x_i=-x_i,\\
	y_i*f_1=0, \quad f_1*y_i=-y_i,\\
	y_i*f_2=-y_i, \quad f_2*y_i=0.
\end{array}
$$
By Lemma \ref{lm 515}, we have
$$
x_i*y_j=-\delta_{ij}f_1, \quad y_i*x_j=- \delta_{ij}f_2.
$$
It remains to calculate the terms in $L_i*L_i$. To approach this goal, we define a wedge product $\wedge: L_i\times L_i\rightarrow L_{i+1}$ given by
$$
u\wedge v:=t_i^{-1}(u)*t_i(v),
$$
for any $u,v \in L_i$. Let $u\in L_1$. It is immediate to get 
\begin{flalign*}
	u\wedge u&=t_1^{-1}(u)*t_1(u)\\
	&=(f_2*u)*(u*f_1)\\
	&=((u*f_1))*u)*f_2\\
	&=f_1*f_2=0
\end{flalign*}
by $\bb  f_2,u*f_1\pp =0, q(u)=0$. By linearizing the equation, we find $u\wedge v=-v\wedge u$ for $u, v\in L_1$. A similar argument can be made for $u,v\in L_2$. Now define a trilinear function $\bb  ~,~ ,~\pp $ on $L_i$ by $\bb  u,v,w\pp :=\bb  u,v\wedge w\pp $. It is an alternating trilinear function since $\bb  u,w,v\pp =\bb  u, w\wedge v\pp =-\bb  u, v\wedge w\pp =-\bb  u,v,w\pp $, and
\begin{flalign*}
	\bb  v,u,w\pp &=\bb  v,u\wedge w\pp \\
	&=\bb  v, t_i^{-1}(u)*t_i(w)\pp \\
	&=\rho(\bb  t_i(w)*v,t_i^{-1}(u)\pp )\\
	&=\bb  t_{i+1}(t_i(w)*v),u\pp \\
	&=\bb  t_i^2(w)*t_i(v),u\pp \\
	&=\bb  w\wedge v,u\pp =-\bb  u, v\wedge w\pp.
\end{flalign*}
We can now calculate the terms in $L_i*L_i$. Consider $x_1*x_2$. We have $\bb  x_1*x_2,x_1\pp =-\bb  x_1*x_2,t_1(x_1)\pp =-\bb  x_1*x_2,x_1*f_1\pp =0$ by $\bb  x_2,f_1\pp =0$. Similarly $\bb  x_1*x_2,x_2\pp =0$. Hence we have $x_1*x_2=b y_3$ for some $b=\bb  x_1*x_2,x_3\pp \in R\lp u\rp$. Multiplying by $y_1$ on the right side, we obtain $(x_1*x_2)*y_1=(b y_3)*y_1$. Since $(x_1*x_2)*y_1+(y_1*x_2)*x_1=\theta(\bb  x_1,y_1\pp )x_2=x_2$, and $y_1*x_2=0$, we have
$$
x_2=\rho(b)(y_3*y_1).
$$
Therefore $b,\rho(b)^{-1}\in R\lp u\rp $, which implies $b\in R\lp u\rp^* $.  Let $b=-1$ (replace $b y_3$ by $-y_3$, and also replace $b^{-1}x_3$ by $-x_3$), We get $x_1*x_2=-y_3$. We can perform similar calculations for the other $x_i*x_j$ and $y_i*y_j$. By using the alternating trilinear form, we obtain

\begin{minipage}{\textwidth}
 \begin{minipage}[t]{0.45\textwidth}
  \centering
     \makeatletter\def\@captype{table}\makeatother\caption{$x_i*x_j$}
       \begin{tabular}{|c|c|c|c|}\hline
	$*$& $x_1$&$x_2$&$x_3$\\
	\hline
	$x_1$& $0$& -$ y_3$ & $ y_2$\\
	\hline
	$x_2$& $ y_3$& 0&-$ y_1$\\
	\hline
	$x_3$& $-y_2$& $ y_1$&0\\
	\hline
	\end{tabular}
  \end{minipage}
  \begin{minipage}[t]{0.45\textwidth}
   \centering
        \makeatletter\def\@captype{table}\makeatother\caption{$y_i*y_j$}
        \begin{tabular}{|c|c|c|c|}\hline
	$*$& $y_1$&$y_2$&$y_3$\\
	\hline
	$y_1$& 0& $- x_3$ & $ x_2$\\
	\hline
	$y_2$& $ x_3$& 0&$-x_1$\\
	\hline
	$y_3$& $- x_2$& $ x_1$&0\\
	\hline
	\end{tabular}
   \end{minipage}
\end{minipage}\\

Therefore, we complete the multiplication table of $L$. By letting $g(e_4)=f_1, g(e_5)=f_2$, and 
\[
\begin{array}{c}
g(e_1)=x_1,\quad g(e_6)=x_2,\quad g(e_7)=x_3,\\
g(e_8)=y_1,\quad g(e_3)=y_2,\quad g(e_2)=y_3.
\end{array}
\]
We obtain $g(e_i)*g(e_j)= g(e_i*e_j)$. So, there exist $g\in L\underline{\mathscr{G}}(R)$ such that $L=g(\LL)$.


 \newpage
  \bibliographystyle{plain}
  \let\itshape\upshape
  \bibliography{ref}

\begin{thebibliography}{10}

\bibitem{BL}
A.~Beauville and Y.~Laszlo.
\newblock Un lemme de descente.
\newblock {\em C. R. Acad. Sci. Paris S\'er. I Math.}, 320(3):335--340, 1995.

\bibitem{BD}
A.~Beilinson and V.~Drinfeld.
\newblock Quantization of hitchin's integrable system and hecke eigensheaves.
\newblock preprint.

\bibitem{G1}
U.~G\"{o}rtz.
\newblock On the flatness of models of certain {S}himura varieties of
  {PEL}-type.
\newblock {\em Math. Ann.}, 321(3):689--727, 2001.

\bibitem{Go2}
U.~G\"{o}rtz.
\newblock On the flatness of local models for the symplectic group.
\newblock {\em Adv. Math.}, 176(1):89--115, 2003.

\bibitem{JA1}
N.~Jacobson.
\newblock {\em Basic algebra. {I}}.
\newblock W. H. Freeman and Company, New York, second edition, 1985.

\bibitem{PK}
M.~Kisin and G.~Pappas.
\newblock Integral models of {S}himura varieties with parahoric level
  structure.
\newblock {\em Publ. Math. Inst. Hautes \'{E}tudes Sci.}, 128:121--218, 2018.

\bibitem{K1}
M.-A. Knus.
\newblock {\em Quadratic and {H}ermitian forms over rings}, volume 294 of {\em
  Grundlehren der Mathematischen Wissenschaften [Fundamental Principles of
  Mathematical Sciences]}.
\newblock Springer-Verlag, Berlin, 1991.
\newblock With a foreword by I. Bertuccioni.

\bibitem{KMRT}
M.-A. Knus, A.~Merkurjev, M.~Rost, and J.-P. Tignol.
\newblock {\em The book of involutions}, volume~44 of {\em American
  Mathematical Society Colloquium Publications}.
\newblock American Mathematical Society, Providence, RI, 1998.
\newblock With a preface in French by J. Tits.

\bibitem{KT}
M.-A. Knus and J.-P. Tignol.
\newblock Triality and algebraic groups of type {${}^3{D}_4$}.
\newblock {\em Doc. Math.}, (Extra vol.: Alexander S. Merkurjev's sixtieth
  birthday):387--405, 2015.

\bibitem{Lu}
G.~Lusztig.
\newblock Singularities, character formulas, and a {$q$}-analog of weight
  multiplicities.
\newblock In {\em Analysis and topology on singular spaces, {II}, {III}
  ({L}uminy, 1981)}, volume 101 of {\em Ast\'{e}risque}, pages 208--229. Soc.
  Math. France, Paris, 1983.

\bibitem{PR3}
G.~Pappas and M.~Rapoport.
\newblock Twisted loop groups and their affine flag varieties.
\newblock {\em Adv. Math.}, 219(1):118--198, 2008.
\newblock With an appendix by T. Haines and Rapoport.

\bibitem{PZ}
G.~Pappas and X.~Zhu.
\newblock Local models of {S}himura varieties and a conjecture of {K}ottwitz.
\newblock {\em Invent. Math.}, 194(1):147--254, 2013.

\bibitem{Se}
J.-P. Serre.
\newblock {\em Galois cohomology}.
\newblock Springer Monographs in Mathematics. Springer-Verlag, Berlin, english
  edition, 2002.
\newblock Translated from the French by Patrick Ion and revised by the author.

\bibitem{SM2}
B.~Smithling.
\newblock Topological flatness of local models for ramified unitary groups.
  {II}. {T}he even dimensional case.
\newblock {\em J. Inst. Math. Jussieu}, 13(2):303--393, 2014.

\bibitem{Sp}
T.A. Springer and F.D. Veldkamp.
\newblock {\em Octonions, {J}ordan algebras and exceptional groups}.
\newblock Springer Monographs in Mathematics. Springer-Verlag, Berlin, 2000.

\bibitem{zhu}
X~Zhu.
\newblock An introduction to affine {G}rassmannians and the geometric {S}atake
  equivalence.
\newblock In {\em Geometry of moduli spaces and representation theory},
  volume~24 of {\em IAS/Park City Math. Ser.}, pages 59--154. Amer. Math. Soc.,
  Providence, RI, 2017.

\end{thebibliography}

\Addresses   
  \end{document}